\documentclass[a4paper,english]{amsart}
\usepackage{mathptmx}

\usepackage[T1]{fontenc}
\usepackage[latin9]{inputenc}
\usepackage{amstext}
\usepackage{amsthm}
\usepackage{amssymb}

\makeatletter

\pdfpageheight\paperheight
\pdfpagewidth\paperwidth

\numberwithin{equation}{section}
\numberwithin{figure}{section}
\theoremstyle{plain}
\newtheorem{thm}{\protect\theoremname}
\theoremstyle{definition}

\theoremstyle{remark}
\newtheorem{rem}[thm]{\protect\remarkname}
\theoremstyle{plain}
\newtheorem{lem}{Lemma}
\theoremstyle{plain}
\newtheorem{cor}{Corollary}
\theoremstyle{plain}
\newtheorem{prop}[thm]{\protect\propositionname}

\usepackage{pgf,tikz}
\usetikzlibrary{arrows}

\usepackage{marvosym}
\usepackage{nicefrac}

\makeatother

\usepackage{babel}
\providecommand{\definitionname}{Definition}
\providecommand{\propositionname}{Proposition}
\providecommand{\remarkname}{Remark}
\providecommand{\theoremname}{Theorem}

\begin{document}
\title{\textsf{Discrete-Time Dynamical Systems Generated by a Quadratic Operator}}
\maketitle
\begin{center}
\textsc{S.K. Shoyimardonov, U.A. Rozikov}
\par\end{center}
\begin{abstract}
In this paper, we examine a specific class of quadratic operators. For these operators, we identified all fixed points and categorized their types in the general case. Our analysis revealed that there are no attractive fixed points except the origin. Additionally, we investigated the global dynamics in the two-dimensional case and generalized several results obtained for lower-dimensional scenarios.
\end{abstract}

\textbf{\textsc{\scriptsize{}Mathematics Subject Classifications
(2010).}}\textsc{\scriptsize{} 34D20 (92D25).}{\scriptsize\par}

\textbf{\textsc{\scriptsize{}{Key words.}}}\textsc{\scriptsize{}
 Quadratic operator, fixed point, invariant set, invariant manifold, stable curve, unstable line.}{\scriptsize\par}

\section{Introduction}

In this paper we study dynamical systems generated by the following quadratic operator:

\begin{equation}
H: x_{k}' =\frac{\theta_kx_k}{2}\left(x_k+2\sum_{\substack{i=1
		 \\ i\neq k}}^{n}x_{i}\right),\quad k=1,\ldots,n.\medskip\\
\label{h2}
\end{equation}
where $\theta_k>0, \forall k=1,2,...,n$ and $H: \mathbb{R}_{+}^n\rightarrow \mathbb{R}_{+}^n.$
This operator has applications in several fields due to its non-linear,
interactive structure. Here are a few potential areas where such an operator can be applied:

1. Sequential hermaphroditism: it is a phenomenon observed in many fish species, where an individual changes sex at some point during its life. In populations exhibiting sequential hermaphroditism, the dynamics can be represented using gonosomal algebra, with $\theta_k$  denoting the rate of sexual inversion. If $\theta_k=1/2, \ \ \forall k=1,2,...,n$, then the model under consideration (as described in equation (\ref{h2})) aligns with the sequential hermaphroditism model (see \cite{Gem}, \cite{RV-15}).

2. Population dynamics and genetics: the operator can describe how the population proportions of different species or genotypes evolve over time. Each \(x_k\) could represent the population size of species \(k\), and the term involving \(\sum_{i \neq k} x_i\) can model interactions or competition between species.
Specifically, the structure is reminiscent of quadratic stochastic operators (QSO), which are used in models of genetic inheritance. These operators describe how the frequencies of genotypes in a population evolve under inheritance laws (see \cite{Lyub-92}, \cite{Rpd}) .

3. Economic systems:  In economic models, \(x_k\) could represent the wealth, output, or capital of sector \(k\), and the summation over other sectors might represent interactions such as investments, production exchanges, or competition between different economic sectors. The operator can be used to model non-linear growth in a system of interdependent sectors where changes in one sector affect others (see \cite{Go}).

4. Statistical physics:  the operator could describe the evolution of a system of interacting particles, where \(x_k\) represents some physical quantity such as energy, spin, or concentration of species in a multi-component system. The non-linear term involving the sum of the other components reflects the interaction between different particles or spins, similar to mean-field theory approaches in spin systems like the Ising or Potts models (see \cite{St}).

5. Dynamical systems and control theory: The operator fits into the realm of non-linear dynamical systems and could be applied to model the evolution of systems where each component interacts with others in a non-linear fashion. It could be useful in understanding equilibria and stability in such systems (\cite{Wi}).

In control theory, it could be used to design control systems where multiple subsystems interact non-linearly, and their collective behavior needs to be regulated.

 6. Ecological Systems: the operator can describe interactions among different species in an ecosystem. The term \(\sum_{i \neq k} x_i\) might represent the combined effect of all other species on species \(k\), where competition, symbiosis, or predation influences growth or decline. This could be used to model predator-prey systems or competitive ecosystems where species populations are interdependent (\cite{Ma}).

 7. Epidemiology:  The operator can also be applied to epidemiological models where \(x_k\) represents the proportion of individuals in a certain state (e.g., susceptible, infected, recovered), and the interaction terms model the transmission dynamics. The sum \(\sum_{i \neq k} x_i\) can represent how individuals in other states influence the spread of disease (\cite{He}, \cite{Mu}).

The paper is organized as follows: In Section 2, we identify all fixed points of the operator (\ref{h2}) and establish conditions for the parameters that determine the types of these fixed points. We prove that one eigenvalue of the Jacobian matrix at the positive fixed point is always equal to 2. In Section 3, we investigate the global dynamics of the operator in the two-dimensional case and generalize some of our findings. In Section 4, we present results concerning invariant curves for the two-dimensional case and generalize our conclusions about the unstable line associated with the positive fixed point. To illustrate our results, we provide some examples.

\section{Fixed points}

Let us find fixed points of $H$, i.e. we solve the following system of
equations
\begin{equation}
\frac{\theta_kx_k}{2}\left(x_k+2\sum_{\substack{i=1 \\ i\neq k}}^{n}x_{i}\right)=x_k,\quad k=1,\ldots,n.\medskip\\
\label{h4}
\end{equation}

Obviously, $x_k=0$ is a solution of the system (\ref{h4}), so $( \underbrace{0,0,...,0}_{n})$ is a fixed point.

Suppose, $x_k\neq0.$ Then we obtain the following system of linear equations:

\begin{equation}
x_k+2\sum_{\substack{i=1 \\ i\neq k}}^{n}x_{i}=\frac{2}{\theta_k},\quad k=1,\ldots,n.\medskip\\
\label{h5}
\end{equation}

\begin{lem}\label{lema} The solution to the equation (\ref{h5}) is given by the following formula:

\begin{equation}
\overline{x}_i=\frac{4\theta_i\sum_{\substack{1\leq i_1<i_2<...<i_{n-2}\leq n \\ i_{r}\neq i}}\theta_{i_1}\theta_{i_2}\ldots\theta_{i_{n-2}}- (4n-6)\prod_{\substack{j=1 \\ j\neq i}}^{n}\theta_j}{(2n-1)\prod_{j=1}^{n}\theta_j}.\\
\label{h6}
\end{equation}

\end{lem}
\begin{proof} We prove this using the method of induction and Cramer's rule. Consider the determinant of the coefficient matrix of unknowns for the system of equations (\ref{h5}):

\[
\begin{vmatrix}
           1 & 2 & \cdots & 2 \\
           2 & 1 & \cdots & 2 \\
           \vdots & \vdots & \vdots & \vdots \\
           2 & 2 & \cdots & 1 \\
         \end{vmatrix}
\]
We reduce to upper triangular form by following the steps below in order:
 add -2 times the first row to all the other rows;
 add -2/3 times the second row to all subsequent rows;
 add -2/5 times the third row to all subsequent rows etc.
Then we have:

\[
\begin{split}
&\begin{vmatrix}
           1 & 2 &2 & \cdots & 2 \\
           2 & 1 &2 & \cdots & 2 \\
           2 & 2 &1 & \cdots & 2 \\
           \vdots & \vdots & \vdots & \vdots& \vdots \\
           2 & 2 &2& \cdots & 1 \\
         \end{vmatrix}= \begin{vmatrix}
           1 & 2 &2 & \cdots & 2 \\
           0 & -3 &-2& \cdots & -2 \\
           0 & -2 &-3& \cdots & -2 \\
           \vdots & \vdots & \vdots & \vdots& \vdots \\
           0 & -2 &-2 & \cdots &-3 \\
         \end{vmatrix}=\\
         &=\begin{vmatrix}
           1 & 2 &2 & \cdots & 2 \\
           0 & -3 &-2& \cdots & -2 \\
           0 & 0 &-\frac{5}{3}& \cdots & -\frac{2 }{3}\\
           \vdots & \vdots & \vdots & \vdots& \vdots \\
           0 & 0 &-\frac{2}{3} & \cdots &-\frac{5}{3} \\
         \end{vmatrix}=\ldots=\begin{vmatrix}
           1 & 2 &2 & \cdots & 2 \\
           0 & -3 &-2& \cdots & -2 \\
           0 & 0 &-\frac{5}{3}& \cdots & -\frac{2 }{3}\\
           \vdots & \vdots & \vdots & \vdots& \vdots \\
           0 & 0 &0 & \cdots &-\frac{2n-1}{2n-3} \\
         \end{vmatrix}.
         \end{split}
\]

Thus,
\begin{equation}\label{det1}
\Delta_n=
\begin{vmatrix}
           1 & 2 & \cdots & 2 \\
           2 & 1 & \cdots & 2 \\
           \vdots & \vdots & \vdots & \vdots \\
           2 & 2 & \cdots & 1 \\
         \end{vmatrix}=(-1)^{n-1}(2n-1).
\end{equation}
 Note that $\overline{x}_i=\frac{\Delta_n^i}{\Delta_n}.$ Now we prove that

\[
\begin{split}
\Delta_n^i&=(-1)^{n-1}\cdot\frac{4\theta_i\sum_{\substack{1\leq i_1<i_2<...<i_{n-2}\leq n \\ i_{r}\neq i}}\theta_{i_1}\theta_{i_2}\ldots\theta_{i_{n-2}}- (4n-6)\prod_{\substack{j=1 \\ j\neq i}}^{n}\theta_j}{\prod_{j=1}^{n}\theta_j}=\\
&=4(-1)^{n-1}\sum_{\substack{j=1 \\ j\neq i}}^{n}\frac{1}{\theta_j}-(-1)^{n-1}(4n-6)\frac{1}{\theta_i}.
\end{split}
\]

If we prove for the case $\Delta_n^1$, then the remaining cases can be handled in a similar way.

 By solving second and third order determinants we get that

   \[
   \Delta_2^1=-\frac{4\theta_1-2\theta_2}{\theta_1\theta_2},\ \  \Delta_3^1=\frac{4\theta_1(\theta_2+\theta_3)-6\theta_2\theta_3}{\theta_1\theta_2\theta_3}.
   \]

Assume that the equality holds for $n=k.$ We have to show that is true for $n=k+1.$ We calculate  $\Delta_{k+1}^1$ by expending last row:

\[
\begin{split}
\Delta_{k+1}^1 & =\begin{vmatrix}
           \frac{2}{\theta_1} & 2 &2 & \cdots & 2 \\
           \frac{2}{\theta_2} & 1 &2 & \cdots & 2 \\
           \frac{2}{\theta_3} & 2 &1 & \cdots & 2 \\
           \vdots & \vdots & \vdots & \vdots& \vdots \\
           \frac{2}{\theta_{k+1}} & 2 &2& \cdots & 1 \\
         \end{vmatrix}=     (-1)^{k+2} \frac{2}{\theta_{k+1}}\cdot \begin{vmatrix}
           2 & 2 &2 & \cdots & 2 \\
           1 & 2 &2 & \cdots & 2 \\
           2 & 1 &2 & \cdots & 2 \\
           \vdots & \vdots & \vdots & \vdots& \vdots \\
           2 & 2 &2& \cdots & 1 \\
         \end{vmatrix}+\\
         &+2(-1)^{k+3} \begin{vmatrix}
           \frac{2}{\theta_1} & 2 &2 & \cdots & 2&2 \\
           \frac{2}{\theta_2} & 2 &2 & \cdots & 2&2 \\
           \frac{2}{\theta_3} & 1 &2 & \cdots & 2&2 \\
           \vdots & \vdots & \vdots & \vdots& \vdots  & \vdots \\
           \frac{2}{\theta_{k}} & 2 &2& \cdots & 1&2 \\
         \end{vmatrix}+2(-1)^{k+4}  \begin{vmatrix}
           \frac{2}{\theta_1} & 2 &2 & \cdots & 2&2 \\
           \frac{2}{\theta_2} & 1 &2 & \cdots & 2&2 \\
           \frac{2}{\theta_3} & 2 &2 & \cdots & 2&2 \\
           \vdots & \vdots & \vdots & \vdots& \vdots  & \vdots \\
           \frac{2}{\theta_{k}} & 2 &2& \cdots & 1&2 \\
         \end{vmatrix}+\cdots + \Delta_{k}^1=\\
         &=(-1)^{k+2}\frac{2}{\theta_{k+1}}\cdot D+2(-1)^{k+3}D_1+2(-1)^{k+4}D_2+...+2(-1)^{2k+1}D_{k-1}+\Delta_{k}^1.
\end{split}
\]

We compute $D$ multiplying the previous column by -1 and adding it to the next column we get the triangular form:
\[
D=\begin{vmatrix}
           2 & 2 &2 & \cdots & 2 \\
           1 & 2 &2 & \cdots & 2 \\
           2 & 1 &2 & \cdots & 2 \\
           \vdots & \vdots & \vdots & \vdots& \vdots \\
           2 & 2 &2& \cdots & 1 \\
         \end{vmatrix}=\begin{vmatrix}
           2 & 0 &0 & \cdots & 0 \\
           1 & 1 &0 & \cdots & 0 \\
           2 & -1 &1 & \cdots & 0 \\
           \vdots & \vdots & \vdots & \vdots& \vdots \\
           2 & 0 &0& \cdots & 1 \\
         \end{vmatrix}=2
\]

Similarly, we calculate $D_1$ multiplying the previous column by -1 and adding it to the next column (except first column) we get:
\[
D_1= \begin{vmatrix}
           \frac{2}{\theta_1} & 2 &2 & \cdots & 2&2 \\
           \frac{2}{\theta_2} & 2 &2 & \cdots & 2&2 \\
           \frac{2}{\theta_3} & 1 &2 & \cdots & 2&2 \\
           \vdots & \vdots & \vdots & \vdots& \vdots  & \vdots \\
           \frac{2}{\theta_{k}} & 2 &2& \cdots & 1&2 \\
         \end{vmatrix}=\begin{vmatrix}
           \frac{2}{\theta_1} & 2 &0 & \cdots & 0&0 \\
           \frac{2}{\theta_2} & 2 &0 & \cdots & 0&0 \\
           \frac{2}{\theta_3} & 1 &1 & \cdots & 0&0 \\
           \vdots & \vdots & \vdots & \vdots& \vdots  & \vdots \\
            \frac{2}{\theta_{k}} & 0 &2& \cdots & 1&0 \\
           \frac{2}{\theta_{k}} & 0 &2& \cdots & -1&1 \\
         \end{vmatrix}=\frac{4}{\theta_1}-\frac{4}{\theta_2}.
\]

If we exchange second and third row in $D_2$ and calculate as in $D_1$ then we get

\[
D_2= -\left(\frac{4}{\theta_1}-\frac{4}{\theta_3}\right)
\]
and so on. Thus,

 \[
D_{r-1}= (-1)^{r-2}\left(\frac{4}{\theta_1}-\frac{4}{\theta_r}\right), \ \ r=2,3,...,k.
\]

To sum up and using the expression for $\Delta_{k}^1$ we obtain
\[
\begin{split}
\Delta_{k+1}^1&=(-1)^{k+2}\frac{4}{\theta_{k+1}}+2(-1)^{k+3}\left(\frac{4}{\theta_1}-\frac{4}{\theta_2}\right)+2(-1)^{k+5}\left(\frac{4}{\theta_1}-\frac{4}{\theta_3}\right)+...+\Delta_{k}^1=\\
&=(-1)^{k}\frac{4}{\theta_{k+1}}+\frac{(-1)^{k-1}8(k-1)}{\theta_1}-\frac{(-1)^{k-1}(4k-6)}{\theta_1}+8(-1)^k\sum_{\substack{j=2}}^{k}\frac{1}{\theta_j}+4(-1)^{k-1}\sum_{\substack{j=2}}^{k}\frac{1}{\theta_j}=\\
&=4(-1)^k\sum_{\substack{j=2}}^{k+1}\frac{1}{\theta_j}+\frac{(-1)^{k-1}(4k-2)}{\theta_1}=4(-1)^k\sum_{\substack{j=2}}^{k+1}\frac{1}{\theta_j}-\frac{(-1)^{k}(4(k+1)-6)}{\theta_1}.
\end{split}
\]

Thus, we complete the proof of Lemma.
\end{proof}

Now we find the remaining fixed points. It is obvious that
$$\left(0,0,...,\frac{2}{\theta_i},...,0\right)\in\mathbb{R}_{+}^{n}$$
also fixed point for any $i=1,2,...,n.$

Let's assume  $r$ coordinates $x_{s_1}, x_{s_2},...,x_{s_r}$ be zero and $n-r$ coordinates $x_{s_{r+1}},...,x_{s_n}$ be non-zero, where $1\leq s_i\leq n$ for any $i=1,...,n.$ Then we have the following system of equations, similar to the system (\ref{h5}):

\begin{equation}
x_{s_k}+2\sum_{\substack{i=r+1 \\ i\neq k}}^{n}x_{s_i}=\frac{2}{\theta_{s_k}},\quad k=r+1,\ldots,n.\medskip\\
\label{h7}
\end{equation}

And the solution also similar to the solution (\ref{h6}):

\begin{equation}
\widetilde{x}_{s_i}=\frac{4\theta_{s_i}\sum_{\substack{r+1\leq i_1<i_2<...<i_{n-r-2}\leq n \\ i_{j}\neq i}}\theta_{s_{i_1}}\theta_{s_{i_2}}\ldots\theta_{s_{i_{n-r-2}}}- (4(n-r)-6)\prod_{\substack{j=r+1 \\ j\neq i}}^{n}\theta_{s_j}}{(2(n-r)-1)\prod_{j=r+1}^{n}\theta_{s_j}}.\\
\label{h8}
\end{equation}

Thus, using Lemma \ref{lema} and above discussion we have proved the following Theorem.

\begin{thm} The operator (\ref{h2}) has the following fixed points:

(i) \ \ $\left(0,0,...,0\right)\in\mathbb{R}_{+}^{n};$

(ii) $\left(0,0,...,\frac{2}{\theta_i},...,0\right)\in\mathbb{R}_{+}^{n}-$ $i^{th}$  coordinate is non-zero for any $i=1,2,...,n;$

(iii) \ \ $(\overline{x}_{1},\overline{x}_{2},...,\overline{x}_{n})-$ all coordinates are non-zero, where $\overline{x}_{i}$ is defined as in (\ref{h6});

(iv) \ \ $r$ coordinates $\widetilde{x}_{s_1},...,\widetilde{x}_{s_r}-$ are zeros, $n-r$ coordinates $\widetilde{x}_{s_{r+1}},...,\widetilde{x}_{s_n}-$ are non-zeros, where $\widetilde{x}_{s_i}$ is defined as in (\ref{h8}) for $i=r+1,...,n.$
\end{thm}

\begin{prop} The operator (\ref{h2}) has $2^n$ fixed points.
\end{prop}

\begin{proof} We can easily get the proof using the property of combinations $\sum_{k=0}^{n}C_k^n=2^n,$  where  $C_k^n$ means the number of fixed points with $k$ zero coordinates.
\end{proof}

\textbf{Type of fixed points}. First, we consider the low-dimensional cases. Let $n=2.$ Then operator (\ref{h2}) has the following form:

\begin{equation}
H: \left\{ \begin{aligned}x_{1}' & =\frac{\theta_1 x_1^2}{2}+\theta_1x_1x_2 \medskip\\
x_{2}' & =\frac{\theta_2 x_2^2}{2}+\theta_2x_1x_2.
\end{aligned}
\right.\label{h9}
\end{equation}

Fixed points are

\[
E_0=(0;0), \ \ E_1=\left(\frac{2}{\theta_1};0\right), \ \ E_2=\left(0; \frac{2}{\theta_2}\right), \ \ \overline{E}=\left(\frac{4\theta_1-2\theta_2}{3\theta_1\theta_2};\frac{4\theta_2-2\theta_1}{3\theta_1\theta_2}\right)
\]

Note that $2\theta_1>\theta_2, 2\theta_2>\theta_1$ for positiveness of coordinates of $\overline{E}.$

\begin{prop}\label{propn2} The following statements hold true:

(a) \ \ the fixed point $E_0=(0;0)$ is an attracting;

(b) $$E_{1}=\left\{\begin{array}{lll}
&{\rm nonhyperbolic}, ~~&{\rm if} \ \  \theta_1=2\theta_2\\
&{\rm repelling}, ~~& {\rm if} \ \  \theta_1<2\theta_2  \\
&{\rm saddle}, ~~& {\rm if}  \ \  \theta_1>2\theta_2
\end{array}\right.$$

(c) $$E_{2}=\left\{\begin{array}{lll}
&{\rm nonhyperbolic}, ~~&{\rm if} \ \  \theta_2=2\theta_1\\
&{\rm repelling}, ~~& {\rm if} \ \  \theta_2<2\theta_1  \\
&{\rm saddle}, ~~& {\rm if}  \ \  \theta_2>2\theta_1
\end{array}\right.$$

(d) \ \ the fixed point $\overline{E}$ is a saddle.

\end{prop}

\begin{proof} The Jacobian of the operator (\ref{h9}) is
\begin{equation}
J(x_1,x_2)=\begin{bmatrix}
\theta_1(x_1+x_2) & \theta_1x_1\\
\theta_2x_2 & \theta_2(x_1+x_2)
\end{bmatrix}.\label{jac}
\end{equation}
Since $J(0,0)$ is a null matrix, its eigenvalues are zero, so we have a proof of the statement $(a)$. The eigenvalues of $J(E_1)$  (resp. $J(E_2)$) are 2 and $\frac{2\theta_2}{\theta_1}$  (resp. 2 and $\frac{2\theta_1}{\theta_2}$),  one can get the proof of $(b)$ (resp. $(c)$).
Before proceeding to the proof of the last assertion, we present the following key lemma.

\begin{lem}[\cite{Cheng}]\label{keylem} Let $F(\lambda)=\lambda^2+B\lambda+C,$ where $B$ and $C$ are two real constants. Suppose $\lambda_1$ and $\lambda_2$ are two roots of $F(\lambda)=0.$ If $F(1)<0,$ then $F(\lambda)=0$ has one root lying in $(1;\infty).$ Moreover, the other root $\lambda$ satisfies $|\lambda|<1$ if and only if $F(-1)>0.$
\end{lem}

Proof of $(d)$. The Jacobian is
\begin{equation}
J(\overline{E})=\begin{bmatrix}
\frac{2(\theta_1+\theta_2)}{3\theta_2} & \frac{4\theta_1-2\theta_2}{3\theta_2}\\
\frac{4\theta_2-2\theta_1}{3\theta_1} & \frac{2(\theta_1+\theta_2)}{3\theta_1}
\end{bmatrix}.
\end{equation}

 Consider the characteristic polynomial  of $J(E_3)$  (as in \cite{Sh}):
\begin{equation}\label{chareq}
F(x)=x^2-\frac{2(\theta_1+\theta_2)^2}{3\theta_1\theta_2}x+\frac{4(\theta_1+\theta_2)^2-4(5\theta_1\theta_2-2\theta_1^2-2\theta_2^2)}{9\theta_1\theta_2}.
\end{equation}
Then
\[
F(1)=\frac{9\theta_1\theta_2-2(\theta_1+\theta_2)^2-4(5\theta_1\theta_2-2\theta_1^2-2\theta_2^2)}{9\theta_1\theta_2}=\frac{(2\theta_1-\theta_2)(\theta_1-2\theta_2)}{3\theta_1\theta_2}<0,
\]
\[
F(-1)=\frac{9\theta_1\theta_2+10(\theta_1+\theta_2)^2-4(5\theta_1\theta_2-2\theta_1^2-2\theta_2^2)}{9\theta_1\theta_2}=\frac{2\theta_1^2+2\theta_2^2+\theta_1\theta_2}{\theta_1\theta_2}>0,
\]

According to the lemma \ref{keylem}, the proof of the proposition is complete.
\end{proof}

Let $n=3.$ Then operator (\ref{h2}) has the following form:

\begin{equation}
H: \left\{ \begin{aligned}x_{1}' & =\frac{\theta_1 x_1^2}{2}+\theta_1x_1x_2+\theta_1x_1x_3 \medskip\\
x_{2}' & =\frac{\theta_2 x_2^2}{2}+\theta_2x_1x_2+\theta_2x_2x_3 \medskip\\
x_{3}' & =\frac{\theta_3 x_3^2}{2}+\theta_3x_1x_3+\theta_3x_2x_3 \medskip\\
\end{aligned}
\right.\label{h10}
\end{equation}

Fixed points are

\[
\begin{split}
& E_0=(0;0;0), \ \ E_1=\left(\frac{2}{\theta_1};0;0\right), \ \ E_2=\left(0; \frac{2}{\theta_2};0\right),
E_3=\left(0;0; \frac{2}{\theta_3};\right),\\
& E_4=\left(0;\frac{4\theta_2-2\theta_3}{3\theta_2\theta_3};\frac{4\theta_3-2\theta_2}{3\theta_2\theta_3}\right), \ \
E_5=\left(\frac{4\theta_1-2\theta_3}{3\theta_1\theta_3};0; \frac{4\theta_3-2\theta_1}{3\theta_1\theta_3}\right), \\
& E_6=\left(\frac{4\theta_1-2\theta_2}{3\theta_1\theta_2}; \frac{4\theta_2-2\theta_1}{3\theta_1\theta_2}; 0\right), \ \ \overline{E}=(\overline{x}_1; \overline{x}_2; \overline{x}_3),
\end{split}
\]
where $\overline{x}_1, \overline{x}_2, \overline{x}_3$ are defined as following:

\begin{equation}
\overline{x}_1=\frac{4\theta_1(\theta_2+\theta_3)-6\theta_2\theta_3}{5\theta_1\theta_2\theta_3}, \ \ \overline{x}_2=\frac{4\theta_2(\theta_1+\theta_3)-6\theta_1\theta_3}{5\theta_1\theta_2\theta_3}, \ \ \overline{x}_3=\frac{4\theta_3(\theta_1+\theta_2)-6\theta_1\theta_2}{5\theta_1\theta_2\theta_3}.
\label{fpn3}
\end{equation}

\begin{prop} The following statements hold true:

(a) \ \ the fixed point $E_0=(0;0;0)$ is an attracting;

(b) \ \ for any $i=1,2,3$
$$E_{i}=\left\{\begin{array}{lll}
&{\rm nonhyperbolic}, ~~&{\rm if} \ \  \theta_i=2\theta_j \ \ {\rm for \ \ some} \ \ j\neq i\\
&{\rm repelling}, ~~& {\rm if} \ \  \theta_i<2\theta_j  \ \ {\rm for \ \ all} \ \ j\neq i \\
&{\rm saddle}, ~~& {\rm if}  \ \  {\rm otherwise}
\end{array}\right.$$

(c) \ \ for any $k=4,5,6$  the fixed point $E_k$ is a saddle;

(d) \ \ if $\theta_i\theta=2,5$ for some $i=1,2,3$ then the fixed point $\overline{E}$ is a nonhyperbolic fixed point, where $\theta=\frac{1}{\theta_1}+\frac{1}{\theta_2}+\frac{1}{\theta_3}$.

\end{prop}

\begin{proof} The Jacobian of the operator (\ref{h10}) is
\begin{equation}
J(x)=\begin{bmatrix}
\theta_1(x_1+x_2+x_3) & \theta_1x_1 & \theta_1x_1 \\
\theta_2x_2 & \theta_2(x_1+x_2+x_3) & \theta_2x_2 \\
\theta_3x_3 & \theta_3x_3 & \theta_3(x_1+x_2+x_3)
\end{bmatrix}.\label{jac3}
\end{equation}
The proof of assertions $(a)$ and $(b)$ is obvious. Consider the eigenvalues of $J(E_4):$
\begin{equation}
J(E_4)=\begin{bmatrix}
\frac{2\theta_1(\theta_2+\theta_3)}{3\theta_2\theta_3} & 0 & 0 \\
\frac{4\theta_2-2\theta_3}{3\theta_3} & \frac{2(\theta_2+\theta_3)}{3\theta_3}  & \frac{4\theta_2-2\theta_3}{3\theta_3} \\
\frac{4\theta_3-2\theta_2}{3\theta_2} & \frac{4\theta_3-2\theta_2}{3\theta_2}  & \frac{2(\theta_2+\theta_3)}{3\theta_2}
\end{bmatrix}.
\end{equation}
One eigenvalue of $J(E_4)$ is $\frac{2\theta_1(\theta_2+\theta_3)}{3\theta_2\theta_3}.$ As in the proof of Proposition \ref{propn2}, for the remaining eigenvalues we can say that one of them lies in $(1;\infty),$ and the other is less that 1 in absolute value. Similarly, the fixed points $E_5$ and $E_6$ are also saddle fixed points. Consider the proof of statement $(d).$ After some simple calculations, the Jacobian matrix at the point $\overline{E}$ has the form:
\begin{equation}
J(\overline{E})=\begin{bmatrix}
\frac{2\theta_1\theta}{5} & \frac{4\theta_1\theta}{5} -2 & \frac{4\theta_1\theta}{5} -2 \\
\frac{4\theta_2\theta}{5} -2 & \frac{2\theta_2\theta}{5}   & \frac{4\theta_2\theta}{5} -2 \\
\frac{4\theta_3\theta}{5} -2 & \frac{4\theta_3\theta}{5} -2  & \frac{2\theta_3\theta}{5}
\end{bmatrix}.
\end{equation}
From this we get the proof of the last statement.
\end{proof}

\begin{prop}\label{propn3} Let $n=3$ and $J(x)$ be a Jacobian, $\overline{E}$ is a fixed point of (\ref{h10}). Then one eigenvalue $\lambda$ of $J(\overline{E})$ is 2 and
$$\overline{E}=\left\{\begin{array}{lll}
&{\rm repelling}, ~~&{\rm if} \ \  |\lambda_{-}|>1, |\lambda_{+}|>1\\
&{\rm saddle}, ~~& {\rm if}  \ \  {\rm otherwise}
\end{array}\right.$$
where $\lambda_{\mp}=\frac{0.4\theta(\theta_1+\theta_2+\theta_3)-2\mp\sqrt{D}}{2},$ $\theta=\frac{1}{\theta_1}+\frac{1}{\theta_2}+\frac{1}{\theta_3},$ $$D=0.16\theta^2(\theta_1+\theta_2+\theta_3)^2-11.2\theta(\theta_1+\theta_2+\theta_3)+1.92\theta^2(\theta_1\theta_2+\theta_1\theta_3+\theta_2\theta_3)+36.$$
\end{prop}

\begin{proof} After some calculations we obtain the following characteristic polynomial for  $J(\overline{E}):$
$$(2-x)^3+[0.4\theta(\theta_1+\theta_2+\theta_3)-6](2-x)^2$$ $$+[1.6\theta(\theta_1+\theta_2+\theta_3)-0.48\theta^2(\theta_1\theta_2+\theta_1\theta_3+\theta_2\theta_3)](2-x).$$
From this we can easily get the proof of the proposition.
\end{proof}
\begin{rem} Of course, there are values of $\theta_i$ such that $\overline{E}$ will be repelling or saddle. For example, if we choose $\theta_1=0.02, \theta_2=0.02, \theta_3=0.1$ then $\lambda_{-}\approx1.12, \lambda_{+}\approx3.04$ or $\theta_1=0.3, \theta_2=0.5, \theta_3=0.4$ then $\lambda_{-}\approx0.64, \lambda_{+}\approx1.12.$
\end{rem}

\textbf{General case.} The Jacobian matrix of the operator (\ref{h2}) has the form:

\begin{equation}
J(x)=\begin{bmatrix}
\theta_1\sum_{i=1}^nx_i & \theta_1 x_1 & \cdots & \theta_1 x_1 \\
\theta_2 x_2  & \theta_2\sum_{i=1}^nx_i & \cdots & \theta_2 x_2 \\
\vdots &  \vdots& \vdots & \vdots \\
\theta_n x_n & \theta_n x_n & \cdots & \theta_n\sum_{i=1}^nx_i
\end{bmatrix}.\label{jacgeneral}
\end{equation}
Using this Jacobian matrix, we prove the following proposition.

\begin{prop} For the operator (\ref{h2}) the following statements hold true:

(a) \ \ the fixed point $(0,0,...,0)$ is an attracting;

(b) \ \ for any $i=\overline{1,n}$ the fixed point
$$\left(0,0,...,\frac{2}{\theta_i},...,0\right)=\left\{\begin{array}{lll}
&{\rm nonhyperbolic}, ~~&{\rm if} \ \  \theta_i=2\theta_j \ \ {\rm for \ \ some} \ \ j\neq i\\
&{\rm repelling}, ~~& {\rm if} \ \  \theta_i<2\theta_j  \ \ {\rm for \ \ all} \ \ j\neq i \\
&{\rm saddle}, ~~& {\rm if}  \ \  {\rm otherwise}
\end{array}\right.$$

(c) \ \ if $\theta_i\theta=\frac{2n-1}{2}$ for some $i=\overline{1,n}$ then the fixed point $(\overline{x}_{1},\overline{x}_{2},...,\overline{x}_{n})-$ is a non-hyperbolic fixed point, where $\theta=\sum_{j=1}^n\frac{1}{\theta_j}$ and  $\overline{x}_{i}$ are defined as in (\ref{h6});

(d) \ \ if $\theta_i\widetilde{\theta}=\frac{2(n-r)-1}{2}$ for some $i=\overline{1,n}$ then the fixed point with $r$ zero coordinates $\widetilde{x}_{s_1},...,\widetilde{x}_{s_r}$ and $n-r$ non-zero coordinates $\widetilde{x}_{s_{r+1}},...,\widetilde{x}_{s_n}-$ is a non-hyperbolic fixed point, where $\widetilde{\theta}=\sum_{j=r+1}^n\frac{1}{\theta_{s_j}}$  and  $\widetilde{x}_{s_i}$ is defined as in (\ref{h8}) for $i=r+1,...,n.$
\end{prop}

\begin{proof} From the Jacobian (\ref{jacgeneral}) the proof of statements $(a)$ and $(b)$ straightforward.  Let $\overline{x}_{i}$ be coordinates of a non-zero fixed point and are defined by (\ref{h6}). Then,  simplifying the Jacobian  (\ref{jacgeneral}) at this fixed point, we get

\begin{equation}
J(x)=\begin{bmatrix}
\frac{2\theta_1\theta}{2n-1} & \frac{4\theta_1\theta}{2n-1}-2 & \cdots & \frac{4\theta_1\theta}{2n-1}-2  \\
\frac{4\theta_2\theta}{2n-1}-2  & \frac{2\theta_2\theta}{2n-1}& \cdots & \frac{4\theta_2\theta}{2n-1}-2 \\
\vdots &  \vdots& \vdots & \vdots \\
\frac{4\theta_n\theta}{2n-1}-2  & \frac{4\theta_n\theta}{2n-1}-2  & \cdots & \frac{2\theta_n\theta}{2n-1}
\end{bmatrix}\label{jacmn}
\end{equation}
and from this the proof of statement $(c)$ easily follows. Similarly, the proof of statement $(d)$ follows by considering $r$ coordinates are zero (for simplicity, first $r$ coordinates), and $n-r$ coordinates are defined as (\ref{h8}). Thus the proof is complete.
\end{proof}

Based on the Proposition \ref{propn3} we conjectured that for any $n\geq2,$ one eigenvalue is always 2 of the Jacobian matrix  at the positive fixed point (for $n=2$ this is easy to check).

\begin{thm}\label{thm1} Let $J(x)$ be a Jacobian of (\ref{h4}) and $\overline{E}$ be a non-zero fixed point whose coordinates are defined by (\ref{h6}). Then for any $n\geq2$ one eigenvalue of $J(\overline{E})$ is 2.
\end{thm}

\begin{proof} The idea is to show that the determinant of a matrix  $J(\overline{E})-2I$ is zero.
Using the property of the determinant, we express the main determinant as a sum of $n+1$ determinants:

\[
\begin{vmatrix}
\frac{2\theta_1\theta}{2n-1}-2 & \frac{4\theta_1\theta}{2n-1}-2 & \cdots & \frac{4\theta_1\theta}{2n-1}-2  \\
\frac{4\theta_2\theta}{2n-1}-2  & \frac{2\theta_2\theta}{2n-1}-2& \cdots & \frac{4\theta_2\theta}{2n-1}-2 \\
\vdots &  \vdots& \vdots & \vdots \\
\frac{4\theta_n\theta}{2n-1}-2  & \frac{4\theta_n\theta}{2n-1}-2  & \cdots & \frac{2\theta_n\theta}{2n-1}-2
\end{vmatrix}=A+B,
\]
where
\[
A=
\begin{vmatrix}
\frac{2\theta_1\theta}{2n-1} & \frac{4\theta_1\theta}{2n-1} & \cdots & \frac{4\theta_1\theta}{2n-1}  \\
\frac{4\theta_2\theta}{2n-1}  & \frac{2\theta_2\theta}{2n-1}& \cdots & \frac{4\theta_2\theta}{2n-1} \\
\vdots &  \vdots& \vdots & \vdots \\
\frac{4\theta_n\theta}{2n-1}  & \frac{4\theta_n\theta}{2n-1}  & \cdots & \frac{2\theta_n\theta}{2n-1}
\end{vmatrix}
\]
and
\[
\begin{split}
B&=
\begin{vmatrix}
-2 & \frac{4\theta_1\theta}{2n-1} & \cdots & \frac{4\theta_1\theta}{2n-1}  \\
-2 & \frac{2\theta_2\theta}{2n-1}& \cdots & \frac{4\theta_2\theta}{2n-1} \\
\vdots &  \vdots& \vdots & \vdots \\
-2  & \frac{4\theta_n\theta}{2n-1}  & \cdots & \frac{2\theta_n\theta}{2n-1}
\end{vmatrix}+
\begin{vmatrix}
\frac{2\theta_1\theta}{2n-1} & -2 & \cdots & \frac{4\theta_1\theta}{2n-1}  \\
\frac{4\theta_2\theta}{2n-1}  & -2& \cdots & \frac{4\theta_2\theta}{2n-1} \\
\vdots &  \vdots& \vdots & \vdots \\
\frac{4\theta_n\theta}{2n-1}  & -2  & \cdots & \frac{2\theta_n\theta}{2n-1}
\end{vmatrix}+\\
&+\cdots+
\begin{vmatrix}
\frac{2\theta_1\theta}{2n-1} & \frac{4\theta_1\theta}{2n-1} & \cdots & -2  \\
\frac{4\theta_2\theta}{2n-1}  & \frac{2\theta_2\theta}{2n-1}& \cdots & -2 \\
\vdots &  \vdots& \vdots & \vdots \\
\frac{4\theta_n\theta}{2n-1}  & \frac{4\theta_n\theta}{2n-1}  & \cdots & -2
\end{vmatrix}=B_1+B_2+...+B_n.
\end{split}
\]

\hfill

Note that all other  determinants are zero because they contain two identical columns which elements are -2. By the (\ref{det1}), we obtain

\[
A=\frac{(-1)^{n-1}2^n\theta^n\prod_{j=1}^{n}\theta_j}{(2n-1)^{n-1}}.
\]

In addition, it is easy to see that

\begin{equation}\label{det2}
\begin{vmatrix}
           2 & 2 & 2 &\cdots & 2 \\
           2 & 1 & 2 &\cdots & 2 \\
           2 & 2 & 1 &\cdots & 2 \\
           \vdots & \vdots & \vdots & \vdots&\vdots \\
           2 & 2 & 2& \cdots & 1 \\
         \end{vmatrix}=(-1)^{n-1}\cdot2.
\end{equation}

Using the determinants (\ref{det1}) and (\ref{det2}), if we calculate the terms of $B,$  then we get
\[
B_1=(-2)\cdot\frac{2^{n-1}\theta^{n-1}}{(2n-1)^{n-1}}\left((-1)^{n-2}(2n-3)\prod_{\substack{j=1 \\ j\neq 1}}^{n}\theta_j+(-1)^{n-1}\cdot2\sum_{\substack{i=1 \\ i\neq 1}}^{n}\prod_{\substack{j=1 \\ j\neq i}}^{n}\theta_j\right).
\]

In $B_2$ we swap the first and second columns, and then the first and second rows and similarly to $B_1,$ we get
 \[
B_2=(-2)\cdot\frac{2^{n-1}\theta^{n-1}}{(2n-1)^{n-1}}\left((-1)^{n-2}(2n-3)\prod_{\substack{j=1 \\ j\neq 2}}^{n}\theta_j+(-1)^{n-1}\cdot2\sum_{\substack{i=1 \\ i\neq 2}}^{n}\prod_{\substack{j=1 \\ j\neq i}}^{n}\theta_j\right).
\]
etc.

If we find the sum of $B_k, k=\overline{1,n}$ then we get

$B=\sum_{k=1}^{n}B_k=(-2)\cdot\frac{2^{n-1}\theta^{n-1}}{(2n-1)^{n-1}}\left((-1)^{n-2}(2n-3)+(-1)^{n-1}(2n-2)\right)\sum_{i=1}^{n}\prod_{\substack{j=1 \\ j\neq i}}^{n}\theta_j=\\
=-\frac{(-1)^{n-1}2^{n}\theta^{n-1}}{(2n-1)^{n-1}}\sum_{i=1}^{n}\prod_{\substack{j=1 \\ j\neq i}}^{n}\theta_j.
$

On the other hand,
\[
\sum_{i=1}^{n}\prod_{\substack{j=1 \\ j\neq i}}^{n}\theta_j=\theta\prod_{j=1}^{n}\theta_j
\]

Therefore, $A+B=0.$ This completes the proof of the theorem.
\end{proof}

This theorem can be applied to all remaining fixed points of the operator except the origin.

\begin{cor} The Jacobian matrix of the operator (\ref{h4}) has an eigenvalue equals to 2 at any fixed point except the origin.
\end{cor}

\begin{cor} The operator (\ref{h4}) does not have any attracting fixed point except the origin.
\end{cor}

\section{global dynamics}

Let's consider the low-dimensional case again with $n=2.$  Recall that the operator (\ref{h9}) has the form:

\[
H: \left\{ \begin{aligned}x_{1}' & =\frac{\theta_1 x_1}{2}(x_1+2x_2) \medskip\\
x_{2}' & =\frac{\theta_2 x_2}{2}(x_2+2x_1).
\end{aligned}
\right.
\]

Fixed points are

\[
E_0=(0;0), \ \ E_1=\left(\frac{2}{\theta_1};0\right), \ \ E_2=\left(0; \frac{2}{\theta_2}\right), \ \ \overline{E}=\left(\frac{4\theta_1-2\theta_2}{3\theta_1\theta_2};\frac{4\theta_2-2\theta_1}{3\theta_1\theta_2}\right)
\]

\begin{lem}\label{lem2} Let $\theta_1, \theta_2\in(0,\infty)$. Then the following sets are invariant with respect to operator (\ref{h9}):

(i) If $\theta_1<2\theta_2$ and $\theta_2<2\theta_1$ then
\[
M_1=\left\{(x_1,x_2)\in \mathbb{R}_{+}^2: x_1+2x_2\leq\frac{2}{\theta_1}, \ \ x_2+2x_1\leq\frac{2}{\theta_2} \right\},
\]

\[
M_2=\left\{(x_1,x_2)\in \mathbb{R}_{+}^2: x_1+2x_2\geq\frac{2}{\theta_1}, \ \ x_2+2x_1\geq\frac{2}{\theta_2} \right\},
\]

(ii) If $\theta_1>2\theta_2$ then
\[
M_3=\left\{(x_1,x_2)\in R_{+}^2: x_1+2x_2\leq\frac{2}{\theta_1} \right\},
\]

\[
M_4=\left\{(x_1,x_2)\in R_{+}^2:  x_2+2x_1\geq\frac{2}{\theta_2} \right\},
\]

(iii) If $\theta_2>2\theta_1$ then
\[
M_5=\left\{(x_1,x_2)\in R_{+}^2: x_2+2x_1\leq\frac{2}{\theta_2} \right\},
\]

\[
M_6=\left\{(x_1,x_2)\in R_{+}^2:  x_1+2x_2\geq\frac{2}{\theta_1} \right\}.
\]
Moreover, for any initial point $(x_1^{0}, x_2^{0})\in \mathbb{R}^2_+$ (except fixed points) we have

\[
\lim_{n\rightarrow\infty}H^{n}\left(x_1^{0}, x_2^{0}\right) =\left\{\begin{array}{lll}
&(0,0), ~~&{\rm if} \ \  (x_1^{0}, x_2^{0})\in M_i, \ \ i=1,3,5\\
&\infty, ~~& {\rm if} \ \ (x_1^{0}, x_2^{0})\in M_j, \ \ j=2,4,6
\end{array}\right.
\]

\end{lem}
\begin{proof}

(i) Let  $(x_1,x_2)\in M_1.$ Then we have

\[
\left\{\begin{array}{lll}
0\leq x_1+2x_2\leq\frac{2}{\theta_1} \\
0\leq x_2+2x_1\leq\frac{2}{\theta_2}
\end{array}\right.   \Rightarrow \left\{\begin{array}{lll}
0\leq x_{1}'\leq x_1 \\
0\leq x_{2}'\leq x_2
\end{array}\right. \Rightarrow \left\{\begin{array}{lll}
0\leq x_{1}'+2x_{2}'\leq\frac{2}{\theta_1} \\
0\leq x_{2}'+2x_{1}'\leq\frac{2}{\theta_2}
\end{array}\right.
\]
thus, $(x_{1}',x_{2}')\in M_1.$ Similarly, it can be shown that the set $M_2$ is an invariant.

(ii) Let $\theta_1>2\theta_2.$ Consider two lines  $l_1: x_1+2x_2=\frac{2}{\theta_1}$ and $l_2: x_2+2x_1=\frac{2}{\theta_2}.$  The intersection of $l_1$ with the line $x_2=0$ (resp. $x_1=0$) occurs at $x_1=\frac{2}{\theta_1}$ (resp. $x_2=\frac{1}{\theta_1}$), while the intersection of $l_2$ with the line $x_2=0$ (resp. $x_1=0$) occurs at $x_1=\frac{1}{\theta_2}$ (resp. $x_2=\frac{2}{\theta_2}$).  Comparing them we get $\frac{2}{\theta_1}<\frac{2}{2\theta_2}=\frac{1}{\theta_2}$  (resp. $\frac{1}{\theta_1}<\frac{1}{2\theta_2}<\frac{2}{\theta_2}$). Thus, in the first quadrant,  the line $x_1+2x_2=\frac{2}{\theta_1}$ is located in below from the line $x_2+2x_1=\frac{2}{\theta_2}.$ So, from $x_1+2x_2\leq\frac{2}{\theta_1}$ we obtain that $x_2+2x_1\leq\frac{2}{\theta_2}$ and from this we have $x_{1}'\leq x_1,$ $x_{2}'\leq x_2$ which gives that the set $M_3$ is an invariant. Similarly, from $x_2+2x_1\geq\frac{2}{\theta_2}$  we have that $x_1+2x_2\geq\frac{2}{\theta_1}$ and then we obtain that the set $M_4$ is also an invariant set.

(iii) This case can be proved in the same way as case (ii).

In addition, if $(x_1^{0}, x_2^{0})\in M_i, \ \ i=1,3,5$  (resp. $(x_1^{0}, x_2^{0})\in M_j, \ \ j=2,4,6$) then both sequences $x_1^{(n)}, x_2^{(n)}$ are decreasing (resp. increasing) and bounded from below, so they have limits. Furthermore, since the sequences $x_1^{(n)}$ and $x_2^{(n)}$ either both decrease or both increase simultaneously, we can infer from the positions of the fixed points $E_1, E_2, \overline{E}$ that the trajectory converges to the fixed point $(0,0)$ if $(x_1^{0}, x_2^{0}) \in M_i$ for $i = 1, 3, 5$. Conversely, if $(x_1^{0}, x_2^{0}) \in M_j$ for $j = 2, 4, 6$, the trajectory diverges to infinity. The proof is complete.
\end{proof}

\begin{thm}\label{thm2} Let $(x_1^{0}, x_2^{0})\in \mathbb{R}^2_+$ be an initial point (except fixed points).  Then the following statements hold true

(a) If  $\theta_1<2\theta_2, \theta_2<2\theta_1$ then there exists an invariant curve $\gamma(x_1)$ passing through fixed points $E_1, E_2$ and $\overline{E}$ such that

\[
\lim_{n\rightarrow\infty}H^{n}\left(x_1^{0}, x_2^{0}\right) =\left\{\begin{array}{lll}
&(0,0), ~~&{\rm if} \ \  0\leq x_1^{0}<\frac{2}{\theta_1}, \ \ 0\leq x_2^{0}< \gamma(x_1) \\
&\overline{E},~~& {\rm if} \ \ 0< x_1^{0}<\frac{2}{\theta_1}, \ \ x_2^{0}= \gamma(x_1) \\
& \infty,~~& {\rm if} \ \ x_1^{0}\geq0, \ \ x_2^{0}>\gamma(x_1)
\end{array}\right.
\]

(b) If  $\theta_1>2\theta_2$ then there exists an invariant curve $\sigma(x_1)$ passing through fixed points $E_1$ and $E_2$ such that

\[
\lim_{n\rightarrow\infty}H^{n}\left(x_1^{0}, x_2^{0}\right) =\left\{\begin{array}{lll}
&(0,0), ~~&{\rm if} \ \  0\leq x_1^{0}<\frac{2}{\theta_1}, \ \ 0\leq x_2^{0}< \sigma(x_1) \\
&E_1,~~& {\rm if} \ \ 0< x_1^{0}<\frac{2}{\theta_1}, \ \ x_2^{0}= \sigma(x_1) \\
& \infty,~~& {\rm if} \ \ x_1^{0}\geq0, \ \ x_2^{0}>\sigma(x_1)
\end{array}\right.
\]

(c) If  $\theta_2>2\theta_1$ then there exists an invariant curve $\delta(x_1)$ passing through fixed points $E_1$ and $E_2$ such that

\[
\lim_{n\rightarrow\infty}H^{n}\left(x_1^{0}, x_2^{0}\right) =\left\{\begin{array}{lll}
&(0,0), ~~&{\rm if} \ \  0\leq x_1^{0}<\frac{2}{\theta_1}, \ \ 0\leq x_2^{0}< \delta(x_1) \\
&E_2,~~& {\rm if} \ \ 0< x_1^{0}<\frac{2}{\theta_1}, \ \ x_2^{0}= \delta(x_1) \\
& \infty,~~& {\rm if} \ \ x_1^{0}\geq0, \ \ x_2^{0}>\delta(x_1)
\end{array}\right.
\]

\end{thm}
\begin{proof} (a) Let $\theta_1<2\theta_2, \theta_2<2\theta_1.$  Note that the positive fixed point $\overline{E}$ is saddle and fixed points $E_1, E_2$  are repelling. According to Lemma \ref{lem2}, the dynamics remains unknown in two triangular regions $S_1, S_2$, as in Fig.\ref{fig1}. From general theory there exists an invariant curve $\gamma(x_1)$ passing through the fixed points $E_1, E_2, \overline{E}$ which is the trajectory starting on it, converges to $\overline{E}$ (stable curve w.r.t $\overline{E}$). Since the trajectory starting from $M_1$ (resp. $M_2$) converges to the origin (resp. diverges to the infinity) we have that $\gamma(x_1)$ is located in $\mathbb{R}^2_+\setminus (M_1\cup M_2)$ (in $S_1\cup S_2$). Obviously, the trajectory that begins below $\gamma(x_1)$ converges to the origin, while the trajectory that starts above $\gamma(x_1)$ diverges to infinity.

(b) Let $\theta_1>2\theta_2.$ In this case the fixed point $E_1=\left(\frac{2}{\theta_1}, 0\right)$ is a saddle while the fixed point $E_2=\left(0, \frac{2}{\theta_2}\right)$ is repelling. Thus, there exists a stable curve $\sigma(x_1)$ of $E_1$ and passing through $E_2$ that the trajectory beginning from below  $\sigma(x_1)$  converges to (0,0) while the trajectory starting above $\sigma(x_1)$ diverges to infinity. The case (c) is similar to case (b).
\end{proof}

\begin{figure}
  \centering
  \includegraphics[width=8cm]{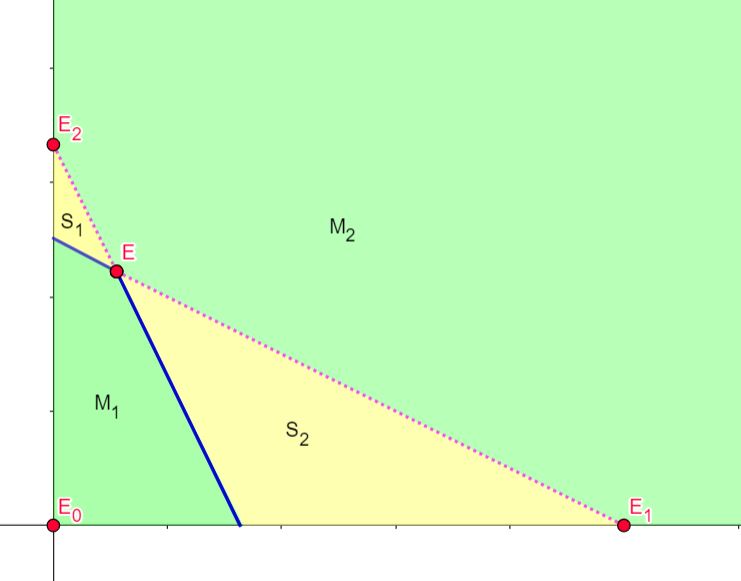}\\
  \caption{In the invariant set $M_1$ the trajectory converges to the origin, in $M_2$ it diverges to infinity.}\label{fig1}
\end{figure}

\begin{rem} For the cases $\theta_1 = 2\theta_2$ or $\theta_2 = 2\theta_1$, the behavior is similar to the cases (b) and (c) of Theorem \ref{thm2}; however, instead of a stable manifold, a center manifold emerges.
\end{rem}

We can generalize the Lemma \ref{lem2} as following:

\begin{prop}\label{prop13} The following sets
\[
\overline{M}_1=\left\{(x_1,x_2,...,x_n)\in \mathbb{R}_{+}^n: x_k+2\sum_{\substack{i=1 \\ i\neq k}}^{n}x_{i}\leq\frac{2}{\theta_k}, \ \ \forall k=1,2,...,n \right\},
\]

\[
\overline{M}_2=\left\{(x_1,x_2,...,x_n)\in \mathbb{R}_{+}^n: x_k+2\sum_{\substack{i=1 \\ i\neq k}}^{n}x_{i}\geq\frac{2}{\theta_k}, \ \ \forall k=1,2,...,n \right\},
\]
are invariant with respect to operator (\ref{h2}). Moreover,  for an initial point $\mathbf{x}^0=(x_1^{0}, x_2^{0},...,x_n^{0})\in \mathbb{R}_{+}^n$ (except fixed points)

\[
\lim_{n\rightarrow\infty}H^{n}\left(\mathbf{x}^0\right) =\left\{\begin{array}{lll}
&(0,0,...,0), ~~&{\rm if} \ \  \mathbf{x}^0\in \overline{M}_1\\
&\infty, ~~& {\rm if} \ \ \mathbf{x}^0\in \overline{M}_2
\end{array}\right.
\]
\end{prop}

\begin{proof} The proof is similar to that of case (i) in Lemma \ref{lem2}.

\end{proof}

\section{invariant manifolds}

\begin{prop}\label{prop25} Let $\theta_1<2\theta_2, \theta_2<2\theta_1.$ The unstable curve of a saddle fixed point $\overline{E}$ is the line
\[
x_2=\frac{2\theta_2-\theta_1}{2\theta_1-\theta_2}x_1
\]
which is passing through the origin and $\overline{E}$.
\end{prop}

\begin{proof} First, we will show that the line $x_2=\frac{2\theta_2-\theta_1}{2\theta_1-\theta_2}x_1$ is an invariant. Let $\frac{x_2}{x_1}=\frac{2\theta_2-\theta_1}{2\theta_1-\theta_2}.$ Then
\[
\frac{x_{2}'}{x_{1}'}=\frac{\theta_2}{\theta_1}\cdot\frac{x_{2}}{x_{1}}\left(\frac{\frac{x_{2}}{x_{1}}+2}{1+\frac{2x_{2}}{x_{1}}}\right)=\frac{x_{2}}{x_{1}}=\frac{2\theta_2-\theta_1}{2\theta_1-\theta_2}.
\]
According to Lemma \ref{lem2}, the trajectory converges to the origin on the part of the line located in $M_1$, and on the other part it diverges to infinity, so the line is unstable with respect to the fixed point $\overline{E}.$
\end{proof}

Finding an exact form of a stable curve is very difficult (almost impossible), but we can find its tangent vector at the fixed point $\overline{E}.$

\begin{prop} The vector $\{1; -\frac{\theta_2}{\theta_1}\}$ is a tangent vector of a stable curve at the fixed point $\overline{E}.$
\end{prop}

\begin{proof} According to Theorem \ref{thm1}, one eigenvalue of a Jacobian at $\overline{E}$ is $\lambda_1=2$. Using this and characteristic polynomial (\ref{chareq}), we can find other eigenvalue $\lambda_2=\frac{2(\theta_1^2+\theta_2^2-\theta_1\theta_2)}{3\theta_1\theta_2}.$ Thus, the eigenvector corresponding to $\lambda_1=2$ is $\left\{\frac{\theta_2-2\theta_1}{\theta_1-2\theta_2}; 1\right\}$ which confirms the Proposition \ref{prop25}.  Similarly, the eigenvector corresponding to $\lambda_2$ is $\{1; -\frac{\theta_2}{\theta_1}\}.$
\end{proof}

Similarly, we can find tangent vectors of the invariant curves $\sigma(x), \delta(x)$ at the fixed points $E_1$ and $E_2$ respectively.

For the concrete values $\theta_1=0,4$ and $\theta_2=0,6$ ($\theta_1<2\theta_2$ and $\theta_2<2\theta_1$ ) we get the following phase portrait for the global dynamics of the operator (\ref{h9}) as in Fig \ref{fig2}. Using approximation, we found polynomial form of the stable curve as following: $\gamma(x)=-0.0046x^5+0.069x^4-0.3987x^3+1.222x^2-2.5674x+3.3333.$

\begin{figure}
  \centering
  \includegraphics[width=9cm]{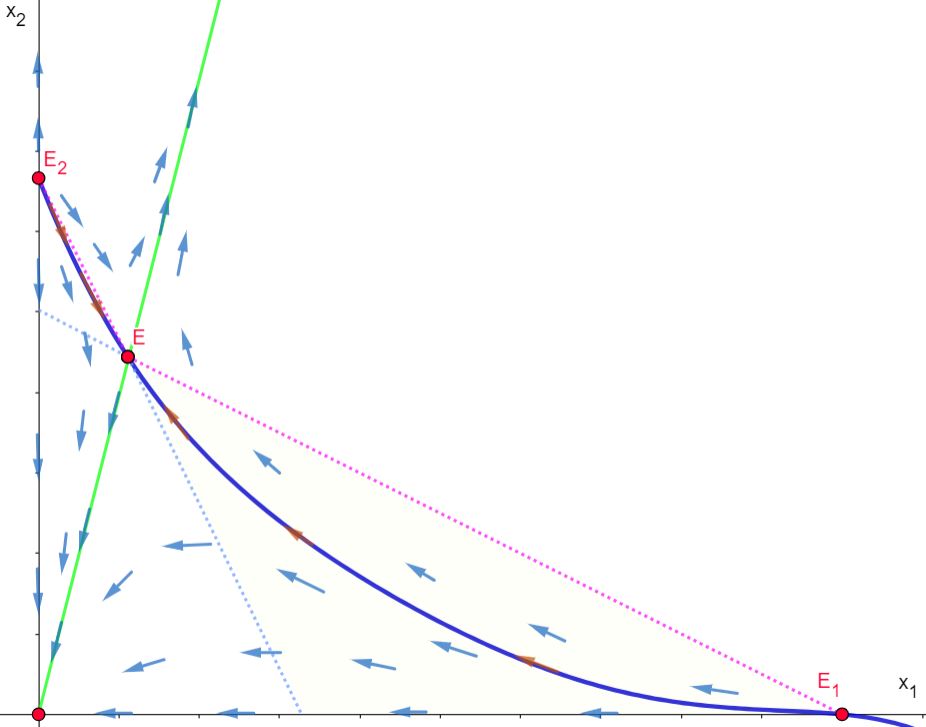}\\
  \caption{Case $\theta_1<2\theta_2, \, \theta_2<2\theta_1.$ The blue curve represents a stable curve, while the green line indicates an unstable curve.}\label{fig2}
\end{figure}

For the values $\theta_1=0,8$ and $\theta_2=0,2$  ($\theta_1>2\theta_2$) we give phase portrait for the global dynamics of the operator (\ref{h9}) as in Fig. \ref{fig3}.

\begin{figure}
  \centering
  \includegraphics[width=5cm]{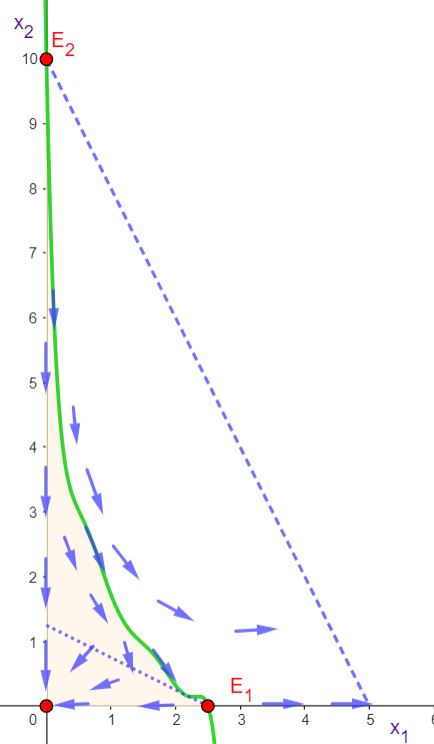}\\
  \caption{Case $\theta_1>2\theta_2.$ The green is a stable curve for the fixed point $E_1$.}\label{fig3}
\end{figure}

%\begin{prop}\label{prop14} Let $n=3.$ The unstable manifold of a saddle fixed point $\overline{E}$ is the line
%\[
%L: x_2=\frac{2\theta_1\theta_2+2\theta_2\theta_3-3\theta_1\theta_3}{2\theta_1\theta_2+2\theta_1\theta_3-3\theta_2\theta_3}x_1, \ \ x_3=\frac{2\theta_1\theta_3+2\theta_2\theta_3-3\theta_1\theta_2}{2\theta_1\theta_2+2\theta_1\theta_3-3\theta_2\theta_3}x_1,
%\]
%which is passing through the origin and $\overline{E}$, where the coordinates of $\overline{E}$ are defined as (\ref{fpn3}).
%\end{prop}
%
%\begin{proof} Let's consider  $\frac{x_{2}'}{x_{1}'}:$
%\[
%\frac{x_{2}'}{x_{1}'}=\frac{\theta_2}{\theta_1}\cdot\frac{x_{2}}{x_{1}}\left(\frac{\frac{x_{2}}{x_{1}}+2+\frac{2x_{3}}{x_{1}}}{1+\frac{2x_{2}}{x_{1}}+\frac{2x_{3}}{x_{1}}}\right).
%\]
%After some simple calculations we get that
%\[
%\frac{x_{2}'}{x_{1}'}=\frac{x_{2}}{x_{1}}, \ \ \frac{x_{3}'}{x_{1}'}=\frac{x_{3}}{x_{1}}.
%\]
%Thus, the line $L$ is an invariant. Furthermore, according to Proposition \ref{prop13}, the trajectory converges to the origin on the part of the line located in $\overline{M}_1$, and it diverges to infinity in $\overline{M}_2$, so the line is unstable with respect to the fixed point $\overline{E}.$
%\end{proof}
%
%We generalize the Proposition \ref{prop14} as following:

 \begin{prop}\label{prop15} For any $n\geq2,$ the unstable manifold of a non-zero saddle fixed point $\overline{E}$ is the $n-$ dimensional line
\[
L_n: \ \ \frac{x_1}{\overline{x}_1}=\frac{x_2}{\overline{x}_2}=\ldots=\frac{x_n}{\overline{x}_n}
\]
which is passing through the origin and $\overline{E}$, where the coordinates of $\overline{E}$ are defined as (\ref{h6}).
\end{prop}

\begin{proof} Let's consider first  $\frac{x_{2}'}{x_{1}'}:$
\[
\frac{x_{2}'}{x_{1}'}=\frac{\theta_2}{\theta_1}\cdot\frac{x_{2}}{x_{1}}\left(\frac{\overline{x}_2+2\sum_{\substack{i=1 \\ i\neq 2}}^{n}\overline{x}_{i}}{\overline{x}_1+2\sum_{\substack{i=2}}^{n}\overline{x}_{i}}\right).
\]
Using (\ref{h5}) we obtain that $\frac{x_{2}'}{x_{1}'}=\frac{x_{2}}{x_{1}}=\frac{\overline{x}_{2}}{\overline{x}_{1}}.$ Similarly, for any $p,q=1,2,...,n$  ($ p\neq q$) we have $\frac{x_{p}'}{x_{q}'}=\frac{x_{p}}{x_{q}}=\frac{\overline{x}_{p}}{\overline{x}_{q}}.$
Thus, the line $L_n$ is an invariant. Furthermore, according to Proposition \ref{prop13}, the trajectory converges to the origin on the part of the line located in $\overline{M}_1$, and it diverges to infinity in $\overline{M}_2$, so the line is unstable with respect to the fixed point $\overline{E}.$
\end{proof}

 \section*{Acknowledgements}
The first author thanks to the International Mathematical Union (IMU-Simons Research Fellowship Program) for
providing financial support of his visit to the Paul Valery University, Montpellier, France. We also thank professor Richard Varro for his helpful discussions.

\address{ S.K. Shoyimardonov$^{a,b}$
\email{\textsc{\footnotesize{}shoyimardonov@inbox.ru}}

\address{ U.A. Rozikov$^{a,b,c}$
\email{\textsc{\footnotesize{}rozikovu@yandex.ru}}
\begin{itemize}
		\item[$^a$] V.I.Romanovskiy Institute of Mathematics,  Uzbekistan Academy of Sciences, 9, Universitet str., 100174, Tashkent, Uzbekistan;
		\item[$^b$]  National University of Uzbekistan named after Mirzo Ulugbek,  4, Universitet str., 100174, Tashkent, Uzbekistan.
		\item[$^c$] Karshi State University, 17, Kuchabag str., 180119, Karshi, Uzbekistan.
\end{itemize}}

\end{document}